\documentclass[12pt]{amsart}
\usepackage{amssymb, amsmath, amsfonts, amsthm, graphics, mathabx, eucal}
\usepackage[bookmarks, colorlinks=true, linkcolor=blue, citecolor=blue, urlcolor=blue]{hyperref}
\usepackage[all]{hypcap}
\usepackage[margin=1in]{geometry}
\usepackage[enableskew,vcentermath]{youngtab} 
\usepackage{tikz}
\usepackage[all]{xy}

\usepackage{mathrsfs}

\usetikzlibrary{arrows}

\newcommand{\newword}[1]{\textbf{\textsf{#1}}}

\newcommand{\into}{\hookrightarrow}

\newcommand{\eset}{\varnothing}

\DeclareMathOperator{\Hom}{Hom}
\DeclareMathOperator{\Sym}{Sym}

\DeclareMathOperator{\Tor}{Tor}
\DeclareMathOperator{\coker}{coker}

\DeclareMathOperator{\rank}{rank}
\DeclareMathOperator{\Spec}{Spec}

\newcommand{\BB}{\mathbb{B}}
\newcommand{\CC}{\mathbb{C}}

\newcommand{\NN}{\mathbb{N}}

\newcommand{\PP}{\mathbb{P}}
\newcommand{\QQ}{\mathbb{Q}}

\renewcommand{\SS}{\mathbb{S}}

\newcommand{\YY}{\mathbb{Y}}
\newcommand{\ZZ}{\mathbb{Z}}

\newcommand{\cE}{\mathcal{E}}
\newcommand{\cF}{\mathcal{F}}

\newcommand{\cM}{\mathcal{M}}

\newcommand{\cO}{\mathcal{O}}

\newcommand{\cQ}{\mathcal{Q}}

\newcommand{\cS}{\mathcal{S}}
\newcommand{\cT}{\mathcal{T}}

\newcommand{\fm}{\mathfrak{m}}

\newcommand{\sR}{\mathscr{R}}

\newcommand{\rE}{\mathrm{E}}
\newcommand{\rH}{\mathrm{H}}

\newtheorem{theorem}{Theorem}[section]
\newtheorem{lemma}[theorem]{Lemma}

\newtheorem{proposition}[theorem]{Proposition}
\newtheorem{conjecture}[theorem]{Conjecture}

\newcommand{\GL}{\mathbf{GL}}
\newcommand{\Gr}{\mathbf{Gr}}
\newcommand{\BS}{\mathrm{BS}}
\newcommand{\ES}{\mathrm{ES}}

\theoremstyle{definition}
\newtheorem{rmk}[theorem]{Remark}
\newenvironment{remark}[1][]{\begin{rmk}[#1] \pushQED{\qed}}{\popQED \end{rmk}}
\newtheorem{eg}[theorem]{Example}
\newenvironment{example}[1][]{\begin{eg}[#1] \pushQED{\qed}}{\popQED \end{eg}}
\newtheorem{defn}[theorem]{Definition}

\numberwithin{figure}{section}
\numberwithin{equation}{section}

\theoremstyle{plain}

\title[Towards Boij--S\"{o}derberg theory for Grassmannians]{Towards Boij--S\"{o}derberg theory for Grassmannians: \\ The case of square matrices}

\date{January 7, 2018}


\author{Nicolas Ford}
\address{
Mathematics Department\\
University of California\\
Berkeley, CA
}
\email{njmford@gmail.com}
\urladdr{\url{http://nicf.net}}

\author{Jake Levinson}
\address{
Mathematics Department\\
University of Michigan\\
Ann Arbor, MI
}
\curraddr{Mathematics Department\\
University of Washington\\
Seattle, WA}
\email{jlev@uw.edu}
\urladdr{\url{http://math.washington.edu/~jlev/}}

\author{Steven V Sam}
\address{
Mathematics Department\\
University of Wisconsin\\
Madison, WI
}
\email{svs@math.wisc.edu}
\urladdr{\url{http://math.wisc.edu/~svs/}}

\begin{document}

\begin{abstract}
We characterize the cone of $\GL$-equivariant Betti tables of Cohen--Macaulay modules of codimension 1, up to rational multiple, over the coordinate ring of square matrices. This result serves as the base case for `Boij--S\"{o}derberg theory for Grassmannians', with the goal of characterizing the cones of $\GL_k$-equivariant Betti tables of modules over the coordinate ring of $k \times n$ matrices, and, dually, cohomology tables of vector bundles on the Grassmannian $\Gr(k,\CC^n)$.
The proof uses Hall's Theorem on perfect matchings in bipartite graphs to compute the extremal rays of the cone, and constructs the corresponding equivariant free resolutions by applying Weyman's geometric technique to certain graded pure complexes of Eisenbud--Fl{\o}ystad--Weyman.
\end{abstract}

\maketitle
\section{Introduction}
\subsection{Ordinary and equivariant Boij--S\"oderberg theory}
Let $M$ be a finitely-generated $\ZZ$-graded module over a polynomial ring $A=\CC[x_1,\ldots,x_n]$. The \newword{Betti table} of $M$ counts the number of generators in each degree of a minimal free resolution of $M$. More precisely, if $M$ has a graded minimal free resolution of the form
\[
M \gets \bigoplus_{d\in\ZZ}A(-d)^{\oplus\beta_{0d}} \gets \cdots \gets \bigoplus_{d\in\ZZ}A(-d)^{\beta_{nd}} \gets 0,
\]
the Betti table of $M$ is the collection of numbers $\beta_{ij}$. Equivalently, $\beta_{ij}$ is the dimension of the degree-$j$ part of the graded module $\Tor_i^A(M,\CC)$.

Boij--S\"{o}derberg theory (initiated in \cite{BS2008}) seeks to characterize the possible Betti tables of graded modules over polynomial rings, with the key insight that it is easier to study these tables only up to positive scalar multiple. The theory has been broadly successful: while the earliest results concerned Betti tables of Cohen--Macaulay modules (stratified by their codimension) \cite{EFW2011, ES2009}, the theory was extended to all modules \cite{BS2012}, to certain modules over multigraded and toric rings \cite{BBEG2012, EE2012}, and more \cite{kummini-sam, 3pts}. For some surveys, see \cite{floystad-expository, survey2}.

In fact, the classification is surprisingly simple. We say a Betti table is \newword{pure} if, for each $i$, exactly one $\beta_{ij}$ is nonzero, i.e., each step of the minimal free resolution is concentrated in a single degree. For each increasing sequence of integers $d_0<d_1<\cdots<d_k$, there is a Cohen--Macaulay $A$-module whose Betti table is pure with the $\beta_{id_i}$'s as the only nonzero entries. Moreover, the resulting Betti table is unique up to a rational multiple, and any purported Betti table is a positive rational multiple of the Betti table of an actual module if and only if it can be written as a positive $\QQ$-linear combination of pure tables \cite{ES2009}. If we bound the degrees that can occur in a Betti table --- that is, bound the $j$'s for which $\beta_{ij}$ may be nonzero --- the Betti tables that fit within these bounds form a rational polyhedral cone. The pure tables then form the extremal rays of this cone.

A key feature of the theory is the discovery that the cone of Betti tables is dual to another cone, consisting of \newword{cohomology tables} of vector bundles and sheaves on projective space. Given such a sheaf $\mathcal{F}$, its cohomology table is the table of numbers $\gamma_{ij} = h^i(\mathcal{F}\otimes\mathcal{O}(j))$. There is a family of nonnegative bilinear pairings between Betti tables of modules and cohomology tables of sheaves, and the inequalities that cut out the cone of Betti tables can all be realized explicitly in terms of this pairing. Consequently, Betti tables yield numerical constraints on the possible cohomology tables on $\PP^n$, and vice versa. Recent work of Eisenbud--Erman has categorified this pairing \cite{EE2012}, realizing it through a functorial pairing between the underlying algebraic objects.

This paper is the beginning of an attempt to generalize this story to $\GL_k$-equivariant modules over a polynomial ring (all $\GL_k$-modules are required to be algebraic representations). Write $R=\CC[x_{11},x_{12},\ldots,x_{kn}]$, the coordinate ring of the affine space of $k\times n$ matrices with the left $\GL_k$ action. In this setting, as we will see in Section \ref{sec:setup}, a minimal free resolution of a finitely-generated equivariant $R$-module comes with an action of $\GL_k$, so in forming our Betti tables we can ask which representations appear at each step of the resolution rather than just which degrees. Specifically, by analogy with the ordinary case, we wish to understand:
\begin{itemize}
\item[(i)] the cone $\BS_{k,n}$ of \newword{$\GL_k$-equivariant Betti tables} of modules supported on the locus of rank-deficient matrices in $\Hom(\CC^k,\CC^n)$, and
\item[(ii)] the cone $\ES_{k,n}$ of \newword{$\GL$-cohomology tables} of vector bundles on the Grassmannian $\Gr(k,\CC^n)$.
\end{itemize}
Both of these constructions will be defined more precisely in Section \ref{sec:setup}.

\begin{remark}
The case $k=1$ is the ordinary Boij--S\"{o}derberg theory of graded modules and vector bundles on projective space since an algebraic action of $\GL_1$ is equivalent to a choice of $\ZZ$-grading (see Remark~\ref{rmk:Z-grading}). We should point out that in this case, \cite{sam-weyman} studies a $\GL_n$-equivariant analogue of Boij--S\"oderberg theory using a Schur-positive analogue of convex cones. In \cite{sam-weyman}, the equivariant Betti table records characters and the Boij--S\"oderberg cone is defined to be closed under ``Schur positive rational functions'' while in the current work, the equivariant Betti table records multiplicities and the cone has an action of the positive rational numbers instead.
\end{remark}

In a later paper, we establish a nonnegative pairing between these tables, extending the pairing of Eisenbud--Schreyer; the hope is that the cones are dual, as they are in ordinary Boij--S\"{o}derberg theory. On the algebraic side, we will restrict to Cohen--Macaulay modules supported everywhere along the rank-deficient locus.

A fundamental base case in ordinary Boij--S\"{o}derberg theory is to understand Betti tables of torsion graded modules over a polynomial ring $\CC[t]$ in one variable. The categorified pairing of Eisenbud--Erman effectively outputs such a module (actually a complex of such modules); as such, the structure of these tables, while very simple, controls the structure of the general Boij--S\"{o}derberg cone and its dual. It is relatively straightforward to write down the inequalities that cut out the corresponding cone, and every inequality cutting out the larger cone of Betti tables comes from pulling back one of these through the pairing mentioned earlier.

This paper is concerned with the corresponding base case, namely, the cone of equivariant modules over the coordinate ring of \emph{square} matrices. This case looks simple at first glance: the modules have codimension 1, and the corresponding Grassmannian is just a point, so there is no dual picture involving vector bundles. Unlike in the graded setting, however, the equivariant base case is already both combinatorially and algebraically interesting. Our main result is the following description of this cone:

\begin{theorem} \label{thm:main}
In the square matrix case, the supporting hyperplanes of the equivariant Boij--S\"{o}derberg cone $\BS_{k,k}$ correspond to \emph{antichains} in the extended Young's lattice $\YY_{\pm}$ of weakly-decreasing integer sequences. Its extremal rays correspond to \emph{pure resolutions} and are indexed by \emph{comparable pairs} of weakly-decreasing integer sequences, $\lambda^{(0)} \subsetneq \lambda^{(1)}$.
\end{theorem}

For a more precise version of this statement, see Theorem~\ref{thm:desc-BS}.

We will exhibit a free resolution to realize each extremal ray, but the construction is nontrivial and relies on existing results of Eisenbud--Fl{\o}ystad--Weyman from ordinary Boij--S\"{o}derberg theory. The proof presented here also depends crucially on the Borel--Weil--Bott theorem, so we do not know if our results hold in positive characteristic.

We expect the description of the cone in Theorem \ref{thm:main} to control the structure of the equivariant Boij--S\"{o}derberg cone in the general case. In particular, the generalized Eisenbud--Schreyer pairing will map the larger cones $\BS_{k,n}$ and $\ES_{k,n}$ to the square-matrix cone. We sketch this construction in Section \ref{sec:setup}.

\subsection{Structure of the paper}

The paper is structured as follows: in Section \ref{sec:setup}, we introduce the relevant notions, namely, equivariant Betti tables and (briefly) $\GL$-cohomology tables for sheaves on Grassmannians. In Section~\ref{sec:square}, we describe the combinatorics of the equivariant Boij--S\"{o}derberg cone for square matrices. In Section~\ref{sec:pure-res}, we show that each extremal ray is realizable, using Weyman's geometric technique. 

\subsection{Acknowledgments}
We thank Daniel Erman, Maria Gillespie, and David Speyer for helpful conversations. Computations in SAGE \cite{sage} and Macaulay2 \cite{M2} have also been very helpful for this work. 
Steven Sam was partially supported by NSF DMS-1500069.

\section{Setup} \label{sec:setup}

Throughout, let $V,W$ be vector spaces over $\CC$ of dimensions $k,n$ respectively, with $k \leq n$. Starting in Section~\ref{sec:square}, we will assume $n=k$.

\subsection{Background}

We will only consider algebraic representations of $\GL(V)$. A good introduction to these notions is \cite{Fulton}. We will also refer the reader to \cite[\S 3]{tca} for a succinct summary (with references) of what we'll need about the representation theory of the general linear group.

The irreducible (algebraic) representations of $\GL(V)$ are indexed by weakly-decreasing integer sequences $\lambda = (\lambda_1, \ldots, \lambda_k)$, where $k = \dim(V)$. We write $\SS_\lambda(V)$ for the corresponding representation, called a \newword{Schur functor}. If $\lambda$ has all nonnegative parts, we write $\lambda \geq 0$ and say $\lambda$ is a \newword{partition}. We often represent partitions by their Young diagrams:
\[
\lambda = (3,1) \longleftrightarrow \lambda = {\tiny \yng(3,1)}.
\]
We partially order partitions and integer sequences by containment:
\[
\lambda \subseteq \mu \text{ if } \lambda_i \leq \mu_i \text{ for all } i.
\]
We write $\YY$ for the poset of all partitions with this ordering, called \newword{Young's lattice}. We write $\YY_{\pm}$ for the set of all weakly-decreasing integer sequences; we call it the \newword{extended Young's lattice}.

If $\lambda$ is a partition, $\SS_\lambda(V)$ is functorial for linear transformations $V \to W$. If $\lambda$ has negative parts, $\SS_\lambda$ is only functorial for isomorphisms $V \xrightarrow{\sim} W$. If $\lambda = (d,0,\ldots,0)$, then $\SS_\lambda(V) = \Sym^d(V)$ and if $\lambda = 1^d=(1,\ldots,1,0,\ldots 0)$ with $d$ 1's, then $\SS_\lambda(V) = \bigwedge^d(V)$. If $\dim V=k$, we'll write $\det(V)$ for the one-dimensional representation $\bigwedge^k(V)=\SS_{1^k}(V)$. We write $K_\lambda(k)$ for the dimension of $\SS_\lambda(\CC^k)$.

We may always twist a representation by powers of the determinant:
\[
\det(V)^{\otimes a} \otimes \SS_{\lambda_1, \ldots, \lambda_k}(V) = \SS_{\lambda_1 + a, \ldots, \lambda_k + a}(V) 
\]
for any integer $a \in \ZZ$. This operation is invertible and can sometimes be used to reduce to considering the case when $\lambda$ is a partition.

By semisimplicity, any tensor product of Schur functors is isomorphic to a direct sum of Schur functors with some multiplicities:
\[
\SS_\lambda(V) \otimes \SS_\mu(V) \cong \bigoplus_\nu \SS_\nu(V)^{\oplus c^\nu_{\lambda, \mu}}.
\]
The $c^\nu_{\lambda, \mu}$ are the \newword{Littlewood--Richardson coefficients}; we won't need to know how they are computed in general, though we will use that if $c^\nu_{\lambda, \mu} \ne 0$ and $\lambda$ is a partition, then $\mu \subseteq \nu$ (and similarly, if $\mu$ is a partition, then $\lambda \subseteq \nu$). Also, by symmetry of tensor products, we have $c^\nu_{\lambda, \mu} = c^\nu_{\mu, \lambda}$. An important special case is \newword{Pieri's rule} when $\lambda = (d)$. In this case, $c^\nu_{(d), \mu} \le 1$ and is nonzero if and only if $\mu \subseteq \nu$ and the complement of $\mu$ in $\nu$ is a horizontal strip, i.e., does not have more than $1$ box in any column. This is equivalent to the interlacing inequalities $\nu_1 \ge \mu_1 \ge \nu_2 \ge \mu_2 \ge \cdots$.

If $R$ is a $\CC$-algebra with an action of $\GL(V)$, and $S$ is any $\GL(V)$-representation, then $S \otimes_\CC R$ is an \newword{equivariant free $R$-module}; it has the universal property
\[\Hom_{\GL(V),R}(S \otimes_\CC R, M) \cong \Hom_{\GL(V)}(S,M)\]
for all equivariant $R$-modules $M$. The basic examples will be the modules $\SS_\lambda(V) \otimes R$. 

\begin{remark}[Gradings and $\GL_1$] \label{rmk:Z-grading}
If $\dim(V) = 1$, the notion of $\GL(V)$-equivariant ring or module is identical to `($\ZZ$-)graded'. In particular, in this case $V^{\otimes d} \otimes R \cong R(-d)$, the rank-1 free module generated in degree $d$.

In general, the modules $\SS_\lambda(V) \otimes R$ are the equivariant analogues of the twisted graded modules $R(-d)$. The analogous notion to `$\NN$-graded ring' is that, as a $\GL(V)$-representation, $R$ should contain only those $\SS_\lambda$ with nonnegative parts.
\end{remark}
\subsection{Equivariant modules and Betti tables}

Fix two vector spaces $V$ and $W$ with $k = \dim(V) \leq \dim(W) = n$. Let $X$ be the affine variety $\Hom(V,W)$, with coordinate ring 
\[
R = \cO_X = \Sym(\Hom(V,W)^*) = \Sym(V \otimes W^*) \cong \CC\bigg[x_{ij} : \begin{aligned}&1 \leq i \leq k, \\ &1 \leq j \leq n\end{aligned}\bigg].
\]
The ring $R$ has actions of $\GL(V)$ and $\GL(W)$. Its structure as a representation is given by the Cauchy identity (see \cite[(3.13)]{tca} for example):
\begin{equation}\label{eqn:cauchy}
R \cong \bigoplus_{\nu \geq 0} \SS_\nu(V) \otimes \SS_\nu(W^*).
\end{equation}
We are primarily interested in the $\GL(V)$ action, though we will use both actions when we construct resolutions in Section~\ref{sec:pure-res}.

The \newword{rank-deficient locus} $\{T : \ker(T) \ne 0\} \subset X$ is an irreducible subvariety of codimension $n-k+1$. Its prime ideal $P_k$ is generated by the $\binom{n}{k}$ maximal minors of the $k \times n$ matrix $(x_{ij})$. When $k=n$, $P_k$ is a principal ideal, generated by the determinant.

Note that the maximal ideal $\fm = (x_{ij})$ of the origin in $X$ and the ideal $P_k$ are $\GL(V)$- and $\GL(W)$-equivariant.

Let $M$ be a finitely-generated $\GL(V)$-equivariant $R$-module. The module $\Tor_R^i(R/\fm,M)$ naturally has the structure of a finite-dimensional $\GL(V)$-representation. We define the \newword{equivariant Betti number} $\beta_{i,\lambda}(M)$ as the multiplicity of the Schur functor $\SS_\lambda(V)$ in this Tor module, i.e.
\begin{equation*}\label{eqn:multiplicity-betti-number}
\Tor_R^i(R/\fm,M)\ \cong\ \bigoplus_\lambda \SS_\lambda(V)^{\oplus \beta_{i,\lambda}(M)} \qquad \text{(as $\GL(V)$-representations).}
\end{equation*}
It is convenient also to define the (equivariant) \newword{rank Betti number} $\widetilde{\beta_{i,\lambda}}(M)$ as the dimension of the $\lambda$-isotypic component, that is,
\begin{equation*}\label{eqn:rank-betti-number}
\widetilde{\beta_{i,\lambda}} := \beta_{i,\lambda} \cdot \dim_\CC(\SS_\lambda(V)) = \beta_{i,\lambda} \cdot K_\lambda(k).
\end{equation*}
By semisimplicity of $\GL(V)$-representations, it is easy to see that any minimal free resolution of $M$ can be made equivariant, so we may instead define $\beta_{i,\lambda}$ as the multiplicity of the equivariant free module $\SS_\lambda(V) \otimes R$ in the $i$-th step of an equivariant minimal free resolution of $M$:
\begin{equation*}
M \leftarrow F_0 \leftarrow F_1 \leftarrow \cdots \leftarrow F_d \leftarrow 0, \text{ where } F_i = \bigoplus_\lambda \SS_\lambda(V)^{\beta_{i,\lambda}(M)} \otimes R,
\end{equation*}
and likewise $\widetilde{\beta_{i,\lambda}}$ is the rank of the corresponding summand as an $R$-module.

\begin{remark} Both definitions $\beta_{i,\lambda}$, $\widetilde{\beta_{i,\lambda}}$ are useful. The Betti number is needed for the pairing with vector bundles, but the rank Betti number is more relevant to the square matrix case and will play the more significant role in this paper.
\end{remark}

\subsubsection{Betti tables and cones}

Let 
\[
\BB_{k,n} = \bigoplus_{i = 0}^{n-k+1} \bigoplus_\lambda \QQ_{i,\lambda}
\]
be a direct sum of copies of $\QQ$, indexed by homological degree $i$ and partition $\lambda$. We think of an element of $\BB_{k,n}$ as an abstract Betti table, that is, a choice of $\beta_{i,\lambda}$ for each $i$ and $\lambda$. Similarly, we write $\widetilde{\BB}_{k,n}$ for the space of abstract rank Betti tables $(\widetilde{\beta_{i,\lambda}})$.

We define the \newword{equivariant Boij--S\"{o}derberg cones} $\BS_{k,n} \subseteq \BB_{k,n}, \widetilde{\BS}_{k,n} \subseteq \widetilde{\BB}_{k,n}$ as the positive linear span of all (multiplicity or rank) Betti tables of finitely generated Cohen--Macaulay modules $M$ supported on the rank-deficient locus $\Spec R/P_k \subset X$. (That is, $M$ for which $\sqrt{\mathrm{ann}(M)} = P_k$.)

\subsection{\texorpdfstring{$\GL$}{GL}-cohomology tables for Grassmannians}

Let $\Gr(k,W)$ denote the Grassmannian variety of $k$-dimensional subspaces in $W$, with tautological exact sequence of vector bundles
\[
0 \to \cS \to W \to \cQ \to 0
\]
where $W$ denotes the trivial rank-$n$ vector bundle and 
\[
\cS = \{(x, U) \in W \times \Gr(k,W) : x \in U\}.
\]
Let $\cE$ be any coherent sheaf on $\Gr(k,W)$. We define the \newword{$\GL$-cohomology table} $\gamma_{i,\lambda}(\cE)$ by
\begin{equation*}\label{eqn:gl-cohomology-number}
\gamma_{i,\lambda}(\cE) := \dim_\CC \rH^i(\cE \otimes \SS_\lambda(\cS)).
\end{equation*}
We let $\ES_{k,n} \subset \bigoplus_i \prod_\lambda \QQ_{i,\lambda}$ be the positive span of such tables.

\begin{remark}
Note that if $k=1$ then $\cS = \cO(-1)$ on the projective space $\PP(W)$. Since this is a line bundle, $\lambda$ can have only one row, say $\lambda = (j)$. Then 
\[\SS_\lambda(\cS) = \Sym^j(\cO(-1)) = \cO(-j),\] so $\gamma_{i,\lambda}(\cE) = \gamma_{i,-j}(\cE)$ is the usual cohomology table of $\cE$ with respect to $\cO(1)$. In general, the $\GL$-cohomology table contains more information than the usual cohomology table with respect to twists by $\cO(1)$; in particular, it determines the class of $\cE$ in K-theory ${\rm K}(\Gr(k,W))$, while the usual table only determines the K-class of $i_*(\cE)$, where $i \colon \Gr(k,W) \into \PP(\bigwedge^k(W))$ is the Pl\"{u}cker embedding.
\end{remark}

\subsection{The numerical pairing}

We briefly discuss the pairing between equivariant Betti tables and $\GL$-cohomology tables. For details, see \cite{FL16}. Let $B = (\beta_{i,\lambda})$ be an equivariant Betti table and $\Gamma = (\gamma_{i,\lambda})$ a $\GL$-cohomology table. We define a \emph{rank} table $\widetilde{\Phi}(B,\Gamma) = \big(\widetilde{\phi_{i,\lambda}}(B,\Gamma)\big)$, for $i \in \ZZ$, by
\[
\widetilde{\phi_{i,\lambda}}(B,\Gamma) = \sum_{p-q = i} \beta_{p,\lambda} \cdot \gamma_{q,\lambda}.
\]
In this definition we do not assume any bounds on $i$, so it is convenient to define the \emph{derived Boij--S\"{o}derberg cone} $\BS_{k,n}^D \subset \bigoplus_{i \in \ZZ} \bigoplus_\lambda \QQ_{i,\lambda}$ as the positive linear span of Betti tables of minimal free equivariant complexes $F^\bullet = \bigoplus_\lambda \SS_\lambda(V)^{\beta_{\bullet,\lambda}} \otimes R$ with homology modules supported in the rank-deficient locus.

\begin{theorem}[{\cite[Theorem 1.13]{FL16}}] \label{thm:numerical-pairing}
The map $\widetilde{\Phi}$ defines a pairing of cones,
\[\widetilde{\Phi} \colon \BS_{k,n}^D \times \ES_{k,n} \to \widetilde{\BS}_{k,k}^D.\]
\end{theorem}
Consequently, any nonnegative linear functional on the cone $\widetilde{\BS}^D_{k,k}$ determines a nonnegative bilinear pairing between equivariant Betti tables and $\GL$-cohomology tables. The extended cone $\widetilde{\BS}_{k,k}^D$ has extremal rays and facets closely resembling those of $\widetilde{\BS}_{k,k}$. See \cite[Section 4.2]{FL16} for an explicit description.

\section{The Boij--S\"{o}derberg cone on square matrices} \label{sec:square}

We now assume $V,W$ are vector spaces of the same dimension $k$, and we describe the cone $\widetilde{\BS}_{k,k} \subset \widetilde{\BB}_{k,k}$. In particular, we would like to know both the extremal rays and the equations of the supporting hyperplanes. The modules $M$ of interest are Cohen--Macaulay of codimension 1, so their minimal free resolutions are just injective maps $F_1 \into F_0$ of equivariant free modules. For $i=0,1$, we put
\[F_i = \bigoplus_\lambda \SS_\lambda(V)^{\beta_{i,\lambda}} \otimes R,\]
and define $\widetilde{\beta_{i,\lambda}} = \beta_{i,\lambda} \cdot K_\lambda(k)$ as in Section~\ref{sec:setup}.

The first observation is that, since $M$ is a torsion module, we must have
\begin{equation} \label{eqn:ranks}
\rank F_0 = \rank F_1, \text{ that is, } \sum_\lambda \widetilde{\beta_{0,\lambda}}(M) = \sum_\lambda \widetilde{\beta_{1,\lambda}}(M).
\end{equation}
Conversely, any injective map of free modules of this form must have a torsion cokernel, which is then Cohen--Macaulay of codimension 1. We will see that the rank condition is the only \emph{linear} constraint on Betti tables, that is, the cone spans this entire linear subspace.

\subsection{Antichain inequalities}
The maps of any minimal complex have positive degree. More precisely, we have the following:

\begin{lemma}[Sequences contract under minimal maps] \label{lem:minimal-partitions}
Let $f \colon \SS_\mu(V) \otimes R \to \SS_\lambda(V) \otimes R$ be any nonzero map. If $\mu = \lambda$, then $f$ is an isomorphism. Otherwise, $\mu \supsetneq \lambda$ and $f$ is minimal.
\end{lemma}
\begin{proof}
This follows from the universal property of equivariant free modules,
\[\Hom_{\GL(V),R}(\SS_\mu(V) \otimes R, \SS_\lambda(V) \otimes R) \cong \Hom_{\GL(V)}(\SS_\mu(V),\SS_\lambda(V) \otimes R). \]
We apply the Cauchy identity \eqref{eqn:cauchy} for $R$ as a $\GL(V)$-representation. We see that
\[\Hom_{\GL(V)}(\SS_\mu(V),\SS_\lambda(V) \otimes R) \cong \bigoplus_{\nu \geq 0} \Hom_{\GL(V)}\big(\SS_\mu(V),\SS_\lambda(V) \otimes \SS_\nu(V) \big) \otimes \SS_\nu(W^*).\]
By the Littlewood--Richardson rule, if $\mu \not\supseteq \lambda$, every summand is 0. If $\mu = \lambda$, the only nonzero summand comes from $\nu = \eset$; we see that the corresponding map is an isomorphism (if nonzero). Finally, if $\mu \supsetneq \lambda$, there is at least one $\nu$ for which the corresponding summand is nonzero and any such $\nu$ must satisfy $|\nu| = |\mu| - |\lambda| > 0$, so the corresponding map of $R$-modules has strictly positive degree (equal to $|\mu|$), hence is a minimal map.
\end{proof}

\begin{remark} \label{rmk:w-convention}
Because the ring $R$ involves $W^*$, not $W$, the analogous computation shows that the sequence labeling $W$ \emph{expands} under a minimal map: that is, a nonzero $\GL(W)$-equivariant map $\SS_\mu(W) \otimes R \to \SS_\lambda(W) \otimes R$ exists if and only if $\mu \subseteq \lambda$ (and is minimal if and only if $\mu \ne \lambda$).
\end{remark}

\noindent In particular, for any fixed $\mu$, a minimal injective map $F_1 \hookrightarrow F_0$ of free modules must inject the summands $\lambda \subseteq \mu$ of $F_1$ into the summands $\lambda \subsetneq \mu$ of $F_0$, and so
\begin{equation} \label{ineq:principal}
\sum_{\lambda \subsetneq \mu} \widetilde{\beta_{0,\lambda}} \geq \sum_{\lambda \subseteq \mu} \widetilde{\beta_{1,\lambda}},
\end{equation}
which gives us some of the inequalities our Betti tables need to satisfy. But in fact these inequalities are not enough. For example, for any pair of partitions $\alpha, \beta$, the summands of $F_1$ given by 
\[
\{ \lambda : \lambda \subseteq \alpha \text{ or } \lambda \subseteq \beta \}
\]
must inject into the summands of $F_0$ given by 
\[
\{ \lambda : \lambda \subsetneq \alpha \text{ or } \lambda \subsetneq \beta \}.
\]
This gives the additional, non-redundant condition
\begin{equation*} \label{ineq:pair}
\sum_{\lambda \subsetneq \alpha \text{ or }\lambda \subsetneq \beta} \widetilde{\beta_{0,\lambda}} \geq \sum_{\lambda \subseteq \alpha \text{ or } \lambda \subsetneq \beta} \widetilde{\beta_{1,\lambda}}.
\end{equation*}

\begin{example}
The following example illustrates that the inequalities \eqref{ineq:principal} are not sufficient. Consider the following rank Betti table, with all entries equal to 1 (shown transposed, with dashed lines indicating containment of partitions):
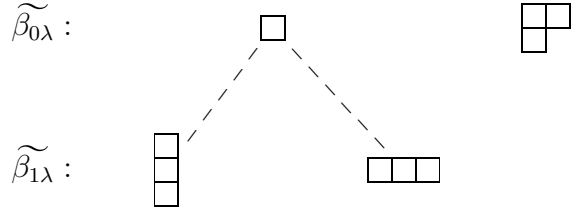
\[\xymatrix{
\widetilde{\beta_{0\lambda}} : && {\tiny \yng(1)} && {\tiny \yng(2,1)} \\
\widetilde{\beta_{1\lambda}} : & {\tiny \yng(1,1,1)} \ar@{--}[ur] && {\tiny \yng(3)} \ar@{--}[ul]
}\]
It is evident that this cannot be the Betti table of a torsion module, since nothing maps to the ${\tiny \yng(2,1)}$ summand. Likewise, the table violates the inequality for the pair ${\tiny \yng(3)}\ ,\ {\tiny \yng(1,1,1)}\ .$ For any single partition $\mu$, however, the inequality \eqref{ineq:principal} is satisfied. (Note that if $\mu$ contains both ${\tiny \yng(3)}$ and ${\tiny \yng(1,1,1)}$, then it strictly contains ${\tiny \yng(2,1)}$.)
\end{example}

The complete set of inequalities is as follows. Recall that if $P$ is a poset, $I \subseteq P$ is an \newword{order ideal} if $x \in I$ and $y \le x$ implies that $y \in I$, i.e., $I$ is a downwards-closed subset. We define the \newword{interior} of $I$ to be the subset
\[
I^\circ = \{ x \in P : x < y \text{ for some  } y \in I \}
\]
of elements strictly contained in $I$. The maximal elements $I \setminus I^\circ$ of $I$ form an \newword{antichain}, that is, they are pairwise incomparable. We have the following:
\begin{lemma}[Antichain inequalities]
Let $\YY_{\pm}$ be the extended Young's lattice and $I \subseteq \YY_{\pm}$ an order ideal. Let $(\beta_{i,\lambda}) \in \BS_{k,k}$ be a Betti table. Then
\begin{equation} \label{ineq:antichains}
\sum_{\lambda \in I^\circ} \widetilde{\beta_{0,\lambda}} \geq \sum_{\lambda \in I} \widetilde{\beta_{1,\lambda}}.
\end{equation}
\end{lemma}

\begin{proof}
Follows from the above discussion.
\end{proof}

\begin{remark}[Inequalities for upwards-closed sets]
It is also the case that, for any \emph{upwards}-closed subset $U \subseteq \YY_{\pm}$, we have a ``dual'' inequality
\begin{equation} \label{ineq:upwards}
\sum_{\lambda \in U} \widetilde{\beta_{0,\lambda}} \leq \sum_{\lambda \in U_\circ} \widetilde{\beta_{1,\lambda}},
\end{equation}
where $U_\circ = \{\lambda \in U : \lambda \ge \mu \text{ for some } \mu \in U\}$ is its upwards-interior.

Algebraically, this corresponds to the following observation: let $(F_1)_{U_\circ}$, $(F_0)_U$ be the summands corresponding to $U_\circ$, $U$. The projection $F_0 \twoheadrightarrow (F_0)_U$ vanishes on the images of the non-${U_\circ}$ summands of $F_1$, so we have a commutative diagram
\[\xymatrix{
F_1 \ar[r]^f \ar[d] & F_0 \ar[d] \\
(F_1)_{U_\circ} \ar[r]^{\bar{f}} & (F_0)_U.
}\]
It follows that $\coker(f) \to \coker(\bar{f})$ is surjective, hence $\coker(\bar{f})$ is also torsion (since $\coker(f)$ is). Consequently, we obtain the desired inequality $\rank (F_1)_{U_\circ} \geq \rank (F_0)_U$.

Alternatively, we may deduce \eqref{ineq:upwards} by subtracting the inequality \eqref{ineq:antichains} with $I = \YY_{\pm} \setminus (U^\circ)$ from the rank equation \eqref{eqn:ranks}, and observing that 
\[
(P \setminus (U_\circ))^\circ \subseteq P \setminus U
\]
holds in any poset $P$. (Note that the complement of an upwards-closed set is downwards-closed.) In particular, given the rank equation \eqref{eqn:ranks}, the ``upwards-facing'' and ``downwards-facing'' inequalities collectively cut out the same cone.
\end{remark}

\subsection{Extremal rays and pure diagrams}

We can construct a very simple Betti table by letting $\lambda^{(0)} \subsetneq \lambda^{(1)}$ be any pair of distinct, comparable partitions. Let $\widetilde{\beta_{0,\lambda^{(0)}}} = \widetilde{\beta_{1,\lambda^{(1)}}} = 1$ and let all other entries of the Betti table be $0$. We call the resulting table $\widetilde{P}(\lambda^{(0)}, \lambda^{(1)})$ a \newword{pure table of type $(\lambda^{(0)},\lambda^{(1)})$}, and any resolution corresponding to a positive multiple of this table a \newword{pure resolution}.

Note that a pure table clearly satisfies all of the antichain inequalities \eqref{ineq:antichains}, as well as the linear constraint \eqref{eqn:ranks}. Moreover, since a Betti table must have at least two nonzero entries, a pure table cannot be written as a nontrivial positive combination of other tables. Any realizable pure table therefore generates an extremal ray of $\widetilde{\BS}_{k,k}$.

We will show in Theorem~\ref{thm:realizable} that every pure table has a realizable scalar multiple. Consequently, \emph{every} pure table generates an extremal ray of the Boij--S\"{o}derberg cone $\widetilde{\BS}_{k,k}$. 

We now show that, \emph{assuming} Theorem \ref{thm:realizable}, every extremal ray is of this form.

\begin{theorem}[Extremal rays] \label{lem:extremal-rays}
Every realizable rank Betti table is a positive $\ZZ$-linear combination of pure rank tables.
\end{theorem}

The proof uses Hall's Theorem on perfect matchings in bipartite graphs. Recall that a \newword{perfect matching} on a graph $G$ is a subset $E' \subseteq E$ of the edges of $G$, such that every vertex of $G$ occurs on exactly one edge from $E'$. We recall the statement of Hall's Matching Theorem (see \cite[Theorem 1.1.3]{matchings} for a proof):

\begin{theorem}[Hall] \label{thm:hall}
Let $G$ be a bipartite graph with left vertices $L$, right vertices $R$ and edges $E$. For a collection of vertices $S$, let $\Gamma(S)$ be the set of neighboring vertices to $S$.

Assume $|L| = |R|$. Then $G$ has a perfect matching if and only if $|\Gamma(S)| \geq |S|$ for all subsets $S \subseteq R$.
\end{theorem}

\begin{proof}[Proof of Theorem \ref{lem:extremal-rays}]
Let $(\widetilde{\beta_{i,\lambda}}) \in \widetilde{\BS}_{k,k}$ be a realizable rank Betti table; by rescaling, we may assume all the entries are integers. We define a bipartite graph $G = (L,R,E)$ as follows: for each $\lambda$, $L$ (resp. $R$) will have $\widetilde{\beta_{0,\lambda}}$ vertices (resp. $\widetilde{\beta_{1,\lambda}}$) labeled $\lambda$. Every vertex labeled $\lambda$ in $L$ is connected to every vertex labeled $\mu$ in $R$ whenever $\lambda \subsetneq \mu$. By the rank condition \eqref{eqn:ranks}, $G$ satisfies $|L| = |R|$.

Observe that a perfect matching on $G$ decomposes $(\beta_{i,\lambda})$ as a $\ZZ$-linear combination of pure tables: each edge $(\lambda^{(0)} \leftarrow \lambda^{(1)})$ in the matching corresponds to a pure table $\widetilde{P}(\lambda^{(0)}, \lambda^{(1)})$. Thus, it suffices to show that $G$ has a perfect matching.

We apply Hall's Theorem (Theorem~\ref{thm:hall}). Let $S \subseteq R$. Observe that if $S$ contains a vertex labeled $\mu$, then without loss of generality, we may assume $S$ contains every vertex labeled $\mu$ and, in addition, every vertex labeled $\mu'$ with $\mu' \subseteq \mu$, since adding these vertices makes $S$ larger but does not change $\Gamma(S)$.

Let $I$ be the order ideal generated by the set of vertex labels appearing in $S$. We see that $|S| = \sum_{\lambda \in I} \widetilde{\beta_{1,\lambda}}$ and $|\Gamma(S)| = \sum_{\lambda \in I^\circ} \widetilde{\beta_{0,\lambda}}$, so the condition $|\Gamma(S)| \geq |S|$ is precisely the antichain inequality \eqref{ineq:antichains} for $I$.
\end{proof}

Thus, assuming Theorem \ref{thm:realizable}, we have shown:

\begin{theorem}[Combinatorial description of $\widetilde{\BS}_{k,k}$] \label{thm:desc-BS}
The cone $\widetilde{\BS}_{k,k} \subset \widetilde{\BB}_{k,k}$ is cut out by the rank equation \eqref{eqn:ranks}, the antichain inequalities \eqref{ineq:antichains} and the conditions $\widetilde{\beta}_{i,\lambda} \geq 0$. Its extremal rays are exactly the rays spanned by the pure tables corresponding to all pairs $\lambda^{(0)} \subsetneq \lambda^{(1)}$ of comparable elements of $\YY_{\pm}$.
\end{theorem}

\begin{remark}[Decomposing Betti tables]
There are efficient algorithms for computing perfect matchings of graphs; see e.g. \cite[\S 1.2]{matchings}. A standard proof of Hall's Theorem implicitly uses the following algorithm, which is inefficient but conceptually clear. Let $\widetilde{B} \in \widetilde{\BS}_{k,k}$ be a Betti table.
\begin{enumerate}
\item[Case 1:] Suppose every antichain inequality is strict. Choose any pure table $\widetilde{P}(\lambda^{(0)},\lambda^{(1)})$ whose entries occur with nonzero values in $\widetilde{B}$. Then
\[\widetilde{B}_{\text{rest}} = \widetilde{B} - \widetilde{P}(\lambda^{(0)},\lambda^{(1)}) \in \widetilde{\BS}_{k,k}.\]
Continue the algorithm on $\widetilde{B}_{\text{rest}}$.
\item[Case 2:] Suppose, instead, there exists an antichain $I$ for which \eqref{ineq:antichains} is an equality. Write \[\widetilde{B} = \widetilde{B}_I + \widetilde{B}_{\text{rest}},\] where $\widetilde{B}_I$ contains all the entries involved in the equality ($\widetilde{\beta_{0,\lambda}}$ for $\lambda \in I^\circ$ and $\widetilde{\beta_{1,\lambda}}$ for $\lambda \in I$). Then both $\widetilde{B}_I \in \widetilde{\BS}_{k,k}$ and $ \widetilde{B}_{\text{rest}} \in \widetilde{\BS}_{k,k}$; continue the algorithm separately for each.
\end{enumerate}
We contrast the algorithm above with the usual algorithm \cite[\S 1]{ES2009} for decomposing graded Betti tables. For graded tables, the decomposition is ``greedy'' and deterministic. It relies on a partial ordering on pure graded Betti tables, which induces a decomposition of the Boij--S\"{o}derberg cone as a simplicial fan. Unfortunately, the natural choices of partial ordering on the equivariant pure tables $P(\lambda^{(0)}, \lambda^{(1)})$ do not yield valid greedy decomposition algorithms. For example, suppose the graph $G$ of Theorem \ref{lem:extremal-rays} consists of a single long path. Compare the following two examples:
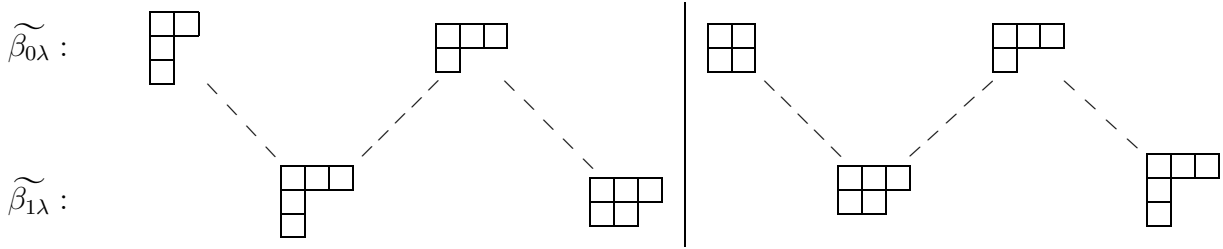
\begin{center}
\begin{tabular}{c|c}
$\xymatrix{
\widetilde{\beta_{0\lambda}} : &  {\tiny \yng(2,1,1)} && {\tiny \yng(3,1)} \\
\widetilde{\beta_{1\lambda}} : && {\tiny \yng(3,1,1)} \ar@{--}[ul] \ar@{--}[ur] && {\tiny \yng(3,2)} \ar@{--}[ul]}$
&
$\xymatrix{
{\tiny \yng(2,2)} && {\tiny \yng(3,1)} \\
& {\tiny \yng(3,2)} \ar@{--}[ul] \ar@{--}[ur] && {\tiny \yng(3,1,1)} \ar@{--}[ul]}$
\end{tabular}
\end{center}
In both cases, $G$ has a unique perfect matching, but whether an edge is used depends on its placement along the path, not just on the partitions labelling its vertices. Hence, an algorithm that (for instance) greedily selects the lexicographically-largest pair $(\lambda^{(0)}, \lambda^{(1)})$ will fail:
in both cases, the lex-largest $\lambda^{(0)}$ is ${\tiny \yng(3,1)} = (3,1)$ and its lex-largest neighbor is $\lambda^{(1)} = {\tiny \yng(3,2)} = (3,2)$. This leads to the (unique) correct matching on the first graph, but fails on the graph to the right.

Similarly, the graph structure of $G$ is a cycle, then $G$ has two perfect matchings, so a deterministic algorithm must have a way of selecting one.

Finally, unlike in the graded case, we do not know a good simplicial decomposition of $\widetilde{\BS}_{k,k}$; it would be interesting to find such a structure.
\end{remark}

\section{Constructing Pure Resolutions} \label{sec:pure-res}

The main theorem of this section is as follows.

\begin{theorem} \label{thm:realizable}
For any partitions $\lambda^{(0)} \subsetneq \lambda^{(1)}$, there exists a torsion, $\GL(V)$-equivariant $R$-module $M$ with minimal free resolution
\[M \leftarrow \SS_{\lambda^{(0)}}(V)^{c_0} \otimes R \leftarrow \SS_{\lambda^{(1)}}(V)^{c_1} \otimes R\leftarrow 0,\]
for some positive integers $c_0, c_1$.
\end{theorem}

We first consider a pair of partitions differing by a box. The same argument works somewhat more generally (see Remark~\ref{rmk:generalizations}), but we restrict to this case for notational simplicity. 

By the Pieri rule, there is a unique $\GL(V) \times \GL(W)$-equivariant $R$-linear map
\begin{equation} \label{eqn:pure-resolution}
\SS_{\lambda^{(0)}}(V) \otimes \SS_{\lambda^{(1)}}(W) \otimes R \leftarrow \SS_{\lambda^{(1)}}(V) \otimes \SS_{\lambda^{(0)}}(W) \otimes R.
\end{equation}

\begin{theorem} \label{thm:linear-case}
The bi-equivariant map \eqref{eqn:pure-resolution} is injective.
\end{theorem}

We postpone the proof to \S\ref{ss:proof-linear-case} and now explain how it implies Theorem~\ref{thm:realizable}.

\begin{proof}[Proof of Theorem~\ref{thm:realizable}]
Let $|\lambda^{(1)}| - |\lambda^{(0)}| = r$. Choose a chain of partitions 
\[
\lambda^{(1)} = \alpha^{(r)} \supsetneq \alpha^{(r-1)} \supsetneq \cdots \supsetneq \alpha^{(0)} = \lambda^{(0)}, \text{ with } |\alpha^{(i)}| = |\lambda^{(0)}| + i \text{ for all } i.
\]
By Theorem~\ref{thm:linear-case}, for $i=1, \ldots, r$, there exists a sequence of bi-equivariant, linear injections
\[
f_i \colon \SS_{\alpha^{(i)}}(V) \otimes \SS_{\alpha^{(i-1)}}(W) \otimes R \hookrightarrow \SS_{\alpha^{(i-1)}}(V) \otimes \SS_{\alpha^{(i)}}(W) \otimes R.
\]
Let $g $ be the composite map
\[\xymatrix{
F_1 = \SS_{\alpha^{(r)}}(V) \otimes \SS_{\alpha^{(r-1)}}(W) \otimes \cdots \otimes \SS_{\alpha^{(1)}}(W) \otimes \SS_{\alpha^{(0)}}(W) \otimes R \ar[d]_{f_r \otimes \mathrm{id} \otimes \cdots \otimes \mathrm{id}} \\
\SS_{\alpha^{(r)}}(W) \otimes \SS_{\alpha^{(r-1)}}(V) \otimes \cdots \otimes \SS_{\alpha^{(1)}}(W)  \otimes \SS_{\alpha^{(0)}}(W) \otimes R \ar[d]_{\mathrm{id} \otimes f_{r-1} \otimes \cdots \otimes \mathrm{id} } \\
\vdots \ar[d]_{\mathrm{id} \otimes \cdots \otimes f_2 \otimes \mathrm{id}} \\
\SS_{\alpha^{(r)}}(W) \otimes \SS_{\alpha^{(r-1)}}(W) \otimes \cdots \otimes \SS_{\alpha^{(1)}}(V) \otimes \SS_{\alpha^{(0)}}(W) \otimes R \ar[d]_{\mathrm{id} \otimes \cdots \otimes \mathrm{id} \otimes f_1} \\
F_0 = \SS_{\alpha^{(r)}}(W) \otimes \SS_{\alpha^{(r-1)}}(W) \otimes \cdots\otimes \SS_{\alpha^{(1)}}(W)  \otimes \SS_{\alpha^{(0)}}(V) \otimes R.
}\]
Clearly $g$ is again injective. Since $\rank(F_1) = \rank(F_0)$, we are done.
\end{proof}

\subsection{Proof of Theorem~\ref{thm:linear-case}} \label{ss:proof-linear-case}

Now we prove Theorem~\ref{thm:linear-case}. To cut down on indices, we write $\lambda$ for the smaller partition and $\mu$ for the larger. We put
\begin{align*}
\mu &= (\mu_1, \ldots, \mu_k), \text{ and we assume } \mu_r > \mu_{r+1}, \\
\lambda &= \mu \text{ except for } \lambda_r = \mu_r - 1.
\end{align*}

The proof relies on the Borel--Weil--Bott theorem and a construction of Eisenbud--Fl{\o}ystad--Weyman. We first review these results, then give an informal summary of the argument, and finally give a proof of the theorem.

\subsubsection{Borel--Weil--Bott and Eisenbud--Fl{\o}ystad--Weyman}

On $\PP(W^*)$, we have the short exact sequence
\[
0 \to \cS \to W \to \cO(1) \to 0.
\]
Note that we are using $W$, not $W^*$. 

Given a permutation $\sigma$, define $\ell(\sigma) = \#\{i < j \mid \sigma(i) > \sigma(j)\}$, the number of inversions.

\begin{theorem}[Borel--Weil--Bott, {\cite[Corollary 4.1.9]{weyman}}]\label{thm:bwb}
Let $\beta = (\beta_1, \ldots, \beta_{k-1})$ be weakly decreasing and let $d \in \ZZ$. The cohomology of $\SS_\beta(\cS)(d)$ is determined as follows. Write
\[(d, \beta_1, \ldots, \beta_{k-1}) - (0, 1, \ldots, k-1) = (a_1, \ldots, a_k).\]
\begin{enumerate}
\item If $a_i = a_j$ for some $i \ne j$, every cohomology group of $\SS_\beta(\cS)(d)$ vanishes.
\item Otherwise, a unique permutation $\sigma$ sorts the $a_i$ into decreasing order, $a_{\sigma(1)} > a_{\sigma(2)} > \cdots > a_{\sigma(k)}.$ Put $\lambda = (a_{\sigma(1)}, \ldots, a_{\sigma(k)}) + (0, 1, \ldots, k-1).$
Then
\[
\rH^{\ell(\sigma)}\big(\SS_\beta(\cS)(d)\big) = \SS_\lambda(W),
\]
and $\rH^i(\SS_\beta(\cS)(d)) = 0$ for $i \ne \ell(\sigma)$.
\end{enumerate}
\end{theorem}

We will also use the following result, on the existence of certain equivariant graded free resolutions.

First, for a partition $\lambda$, we say $(i,j)$ is an \newword{outer border square} if $(i,j) \notin \lambda$ and $(i-1,j-1) \in \lambda$ (or $i=1$ or $j=1$), as in the $*$'s below, for $\lambda = (3,1)$:
\newcommand{\mydots}{\ \ast \ \cdots}
\[
{\tiny \young(\hfil\hfil\hfil\ast\ast\ast\mydots,\hfil\ast\ast\ast,\ast\ast,\ast,\ast,\vdots)}
\]

Let $\alpha$ be a partition with $k$ parts, and let $\alpha' \supsetneq \alpha$ be obtained by adding at least one border square in row $1$, and all possible border squares in rows $2, \ldots, k$. Let $\alpha^{(0)} = \alpha$, and for $i=1, \ldots, k$, let $\alpha^{(i)}$ be obtained by adding the chosen border squares only in rows $1, \ldots, i$.

\begin{theorem}[{\cite[Theorem 3.2]{EFW2011}}] \label{thm:efw}
Let $E$ be a $k$-dimensional complex vector space and $R = \Sym(E)$ its symmetric algebra. There is a finite, $\GL(E)$-equivariant $R$-module $M$ whose equivariant minimal free resolution is, with $\alpha^{(i)}$ defined as above,
\[
F_0 \leftarrow F_1 \leftarrow \cdots \leftarrow F_k\leftarrow 0, \qquad F_i = \SS_{\alpha^{(i)}}(E) \otimes R.
\]
\end{theorem}

Since the construction is equivariant, it works in families:

\begin{theorem}\label{thm:EFW-sheaf}
Let $X$ be a complex variety and $\cE$ a rank $k$ vector bundle over $X$. Let $\cE^* \to X$ be the dual bundle. There is a sheaf $\cM$ of $\cO_{\cE^*}$-modules with a locally-free resolution
\[
\cF_0 \leftarrow \cF_1 \leftarrow \cdots \leftarrow \cF_k\leftarrow 0, \qquad \cF_i = \SS_{\alpha^{(i)}}(\cE) \otimes \cO_{\cE^*}.
\]
\end{theorem}
This follows by applying the EFW construction to the sheaf of algebras $\cO_{\cE^*} = \Sym(\cE)$. The resolved sheaf $\cM$ is locally given by $M$ above. Note that $\cM$ is coherent as an $\cO_X$-module, though we will not need this.

\begin{remark}
The construction we presented is also a direct corollary of a special case Kostant's version of the Borel--Weil--Bott theorem, for example see \cite[\S 6]{questions-BS} for some discussion and references. We expect that other cases of Kostant's theorem are relevant for constructing complexes in the non-square matrix case.
\end{remark}

\subsubsection{Informal summary of the argument}
We have fixed $\lambda \subsetneq \mu$, a pair of partitions differing by a box. There is a unique Eisenbud--Fl{\o}ystad--Weyman (EFW) complex with, in one step, a linear differential of the form
\[\SS_{\mu}(E) \otimes \Sym(E) \to \SS_{\lambda}(E) \otimes \Sym(E).\]
The rest of the complex is uniquely determined by this pair of shapes, and is functorial in $E$. We `sheafify' the complex, lifting it to a complex of modules over the algebra $\Sym(V \otimes \cO(-1))$ on $\PP(W^*)$, with terms of the form
\[
\cO(-d_i) \otimes \SS_\alpha(V) \otimes \Sym(V \otimes \cO(-1)),
\]
We twist so that the $0$-th term has degree $d=\mu_1$, base change along the flat extension $\Sym(V \otimes \cO(-1)) \hookrightarrow \Sym(V \otimes W^*)$, and finally tensor through by the vector bundle $\SS_\beta(\cS)$, where $\cS \subset W \times \PP(W^*)$ is the tautological rank-$(k-1)$ subbundle, and $\beta$ is chosen so that all the terms of the resulting complex \emph{except} the desired pair have no cohomology. (That is, $\SS_\beta(\cS)$ has supernatural cohomology with roots at each of the other $d_i$'s.) Finally, we obtain the desired map from the hypercohomology spectral sequence for the complex.

\begin{example}
Let $k=4$ and let $\lambda = (6,1,1,0)$, $\mu = (6,2,1,0) = {\tiny \young(\hfil\hfil\hfil\hfil\hfil\hfil,\hfil\star,\hfil)}$ (the added box is starred). Working on $\PP(W^*)$, the corresponding locally free resolution of sheaves (with the twisting degrees indicated) is, after twisting and base-changing,
\[
\tiny \yng(1,1,1)(6) \leftarrow \yng(6,1,1)(1) \xleftarrow{\ \star \ } \yng(6,2,1) \leftarrow \yng(6,2,2)(-1) \leftarrow \yng(6,2,2,2)(-3),
\]
where $\alpha(d)$ stands for the sheaf $\cO(d) \otimes \SS_\alpha(V) \otimes \Sym(V \otimes W^*)$ on $\PP(W^*)$. The desired linear differential is marked with a $\star$.

We put $\beta = (7,1,0)$ and tensor through by $\SS_\beta(\cS)$ (note that $\cS$ has rank 3). Observe that $\SS_\beta(\cS)(d)$ has no cohomology when $d \in \{6,-1,-3\}$, but that
\begin{align*}
\rH^1\big(\SS_\beta(\cS)(1) \big) = \SS_{621}(W), \qquad 
\rH^1\big(\SS_\beta(\cS) \big) = \SS_{611}(W).
\end{align*}
Consequently, the $\rE_1$ page of the hypercohomology spectral sequence only has the terms
\[
\SS_{611}(V) \otimes \SS_{621}(W) \otimes R \leftarrow \SS_{621}(V) \otimes \SS_{611}(W) \otimes R,
\]
in the second and third columns. (The left term is along the main diagonal.) The spectral sequence, run the other way, collapses with $\rH^i(\cM)$ in the first column, where $\cM$ is the sheaf resolved by this complex. We see that the only nonvanishing term can be $\rH^0(\cM)$, giving an exact sequence of $R$-modules
\[
0 \leftarrow \rH^0(\cM) \leftarrow \SS_{611}(V) \otimes \SS_{621}(W) \otimes R \leftarrow \SS_{621}(V) \otimes \SS_{611}(W) \otimes R \leftarrow 0. \qedhere
\]
\end{example}

\begin{remark} \label{rmk:generalizations}
There are two easy ways to generalize the construction that we have sketched above. First, in the map marked $\star$ above, there is no reason to assume that the two partitions differ by a single box, and the same construction allows them to differ by multiple boxes as long as they are in the same row. In this case, the Pieri rule still implies that the map \eqref{eqn:pure-resolution} is unique up to scalar.

Second, in the above example we chose $\beta$ so that $\SS_\beta(\cS)(d)$ has no cohomology for all $d$ besides the twists appearing in the target and domain of a single differential (in this case, the one marked $\star$). Alternatively, we could choose $\beta$ so that $\SS_\beta(\cS)(d)$ has no cohomology for all but two of the terms in the complex (not necessarily consecutive terms). The end result is also a map of the form \eqref{eqn:pure-resolution} where $\lambda^{(0)}$ and $\lambda^{(1)}$ differ by a connected border strip. In general the map \eqref{eqn:pure-resolution} is not unique up to scalar, however.
\end{remark}

\subsubsection{Combinatorial setup}

We define shapes $\alpha^{(i)}$, $i=0, \ldots, k$, as follows. Consider the squares formed by:

\begin{itemize}
\item the inner border strip of $\mu$ inside rows $1, \ldots, r-1$,
\item the rightmost square in row $r$ of $\mu$,
\item the outer border strip of $\mu$ outside rows $r+1, \ldots, k$.
\end{itemize}

(See Figure \ref{fig:border-strip}.) Then $\alpha = \alpha^{(0)}$ is obtained by deleting all these squares, and $\alpha^{(i)}$ is obtained by including those squares in rows $1, \ldots, i$. Clearly, $\alpha^{(r)} = \mu$ and $\alpha^{(r-1)} = \lambda$.

\begin{figure}
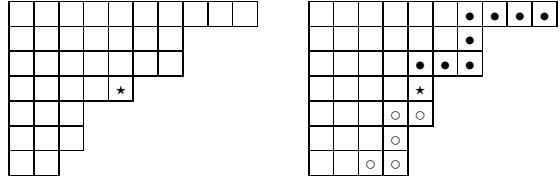

\[\scalebox{.7}{\young(\hfil\hfil\hfil\hfil\hfil\hfil\hfil\hfil\hfil\hfil,\hfil\hfil\hfil\hfil\hfil\hfil\hfil,\hfil\hfil\hfil\hfil\hfil\hfil\hfil,\hfil\hfil\hfil\hfil\star,\hfil\hfil\hfil,\hfil\hfil\hfil,\hfil\hfil) \qquad \young(\hfil\hfil\hfil\hfil\hfil\hfil\bullet\bullet\bullet\bullet,\hfil\hfil\hfil\hfil\hfil\hfil\bullet,\hfil\hfil\hfil\hfil\bullet\bullet\bullet,\hfil\hfil\hfil\hfil\star,\hfil\hfil\hfil\circ\circ,\hfil\hfil\hfil\circ,\hfil\hfil\circ\circ)}\]
\caption{\footnotesize {\bf Left:} The partition $\mu$; the starred box in the 4th row is removed to form $\lambda$. {\bf Right:} The outer strip is formed by connecting the \emph{inner border strip} ($\bullet$) in rows 1 to $r-1$ to the \emph{outer border strip} ($\circ$) outside rows $r+1, \ldots, k$. The empty squares form $\alpha^{(0)}$; then $\alpha^{(i)}$ is obtained by adding all marked squares ($\bullet, \ast, \circ$) up to row $i$. Note that $\alpha^{(3)} = \lambda$ and $\alpha^{(4)} = \mu$.}
\label{fig:border-strip}
\end{figure}

\noindent Let $e_i$ be the number of border squares in row $i$, so
\[
e_i = \begin{cases} 1 + \mu_i - \mu_{i+1} & \text{ for } i = 1, \ldots, r-1, \\ 1 & \text{ for } i=r, \\ 1 + \lambda_{i-1} - \lambda_i & \text{ for } i = r+1, \ldots, k. \end{cases}
\]

We define a modified (and negated) partial sum,
\[
d_i := \mu_1 - (e_1 + \cdots + e_i) = 
\begin{cases} \mu_{i+1} - i & \text{ if } i = 0, \ldots, r-1, \\ 
\mu_r-r & \text{ if } i=r, \\ 
\mu_i - i + 1& \text{ if } i = r+1, \ldots, k.\end{cases}
\]

Finally, we write $\beta = (\beta_1, \ldots, \beta_{k-1})$ for the unique partition such that, on $\PP(W^*)$, the vector bundle $\SS_\beta(\cS)(d)$ has no cohomology for each $d_i$, $i \in \{0, \ldots, k\} \setminus \{r-1,r\},$ where $\cS \subset W \times \PP(W^*)$ is the tautological rank-$(k-1)$ subbundle. By Borel--Weil--Bott, this determines $\beta$ uniquely by
\[\beta - (1, \ldots, k-1) = (d_0, \ldots, d_{r-2}, d_{r+1}, \ldots, d_k).\]
With these choices, we check:
\begin{lemma} \label{lem:bwb-computation}
For $d = d_r$, the only nonvanishing cohomology of $\SS_\beta(\cS)(d)$ on $\PP(W^*)$ is $\rH^{r-1} = \SS_\mu(W)$. For $d = d_{r-1}$, the only nonvanishing cohomology is $\rH^{r-1} = \SS_\lambda(W)$.
\end{lemma}
\begin{proof}
We apply Borel--Weil--Bott: we have to sort
\[(d,\beta_1, \ldots, \beta_k) - (0,1, \ldots, k-1) = (d,d_0, \ldots, d_{r-2}, d_{r+1}, \ldots, d_k).\]
For $d = d_{r-1}$ or $d_r$, sorting takes $r-1$ swaps, so in both cases $\rH^{r-1}$ is nonvanishing. To see that the cohomology group is $\SS_\mu(W)$ for $d_{r-1}$ and $\SS_\lambda(W)$ for $d_r$, we must check that
\begin{align*}
\mu &= (d_0, \ldots, d_{r-2}, d_{r-1}, d_{r+1}, \ldots, d_k) + (0,1, \ldots, k-1), \\
\lambda &= (d_0, \ldots, d_{r-2}, d_r, d_{r+1}, \ldots, d_k) + (0,1, \ldots, k-1)
\end{align*}
These are clear from the computation above.
\end{proof}

\subsubsection{The proof of Theorem \ref{thm:linear-case}}
Let $\alpha^{(i)}$ and $d_i$ be defined as above. Consider the projective space $\PP(W^*)$, with tautological line bundle $\cO(-1) \subset W^*$ and rank-$(k-1)$ bundle $\cS \subset W$. Set $\xi := V \otimes \cO(-1)$. By Theorem~\ref{thm:EFW-sheaf}, we have an exact complex
\[
F_0 \leftarrow \cdots \leftarrow F_i \leftarrow F_{i+1} \leftarrow \cdots, \text{ where } F_i = \SS_{\alpha^{(i)}}(\xi) \otimes \Sym(\xi).
\]
Note that
\[
\SS_\lambda(\xi) = \SS_\lambda(V) \otimes \cO(-|\lambda|).
\]
For legibility, we write $\cO(-\lambda)$ for $\cO(-|\lambda|)$. Thus, we have a locally free resolution
\[\SS_{\alpha^{(0)}}(V) \otimes \cO(-\alpha^{(0)}) \otimes \Sym(\xi) \leftarrow \cdots \leftarrow \SS_{\alpha^{(i)}}(V) \otimes \cO(-\alpha^{(i)}) \otimes \Sym(\xi) \leftarrow \cdots \]
of sheaves of $\Sym(\xi)$-modules.

Next, let $\sR = \cO_{\PP(W^*)} \otimes \Sym(V \otimes W^*)$. Observe that $\Sym(\xi) \hookrightarrow \sR$ is a flat ring extension (locally it is an inclusion of polynomial rings). Now base change to $\sR$, which preserves exactness. Finally, we tensor by $\SS_\beta(\cS) \otimes \cO(\alpha^{(0)} + \mu_1)$. Our final complex has terms
\[
\SS_{\alpha^{(i)}}(V) \otimes \SS_\beta(\cS)(d_i) \otimes \sR.
\]
Let $\cM$ be the sheaf resolved by the complex.

We run the hypercohomology spectral sequence. Running the horizontal maps first, we see that the sequence collapses on the $\rE_2$ page with $\rH^i(\cM)$ in the leftmost column. Running the sequence the other way, the $\rE_1$ page has terms
\[
\rH^q(\SS_{\alpha^{(p)}}(V) \otimes \SS_\beta(\cS)(d_p) \otimes \sR) = \SS_{\alpha^{(p)}}(V) \otimes \rH^q\big(\SS_\beta(\cS)(d_p)\big) \otimes \sR.
\]
(We emphasize that $\SS_{\alpha^{(p)}}(V)$ and $\sR$ are trivial bundles.) By construction, the middle factor is zero unless $p = r-1, r$, where by Lemma \ref{lem:bwb-computation} the nonvanishing cohomology is $\rH^{r-1}$, with
\begin{align*}
\rH^{r-1}\big(\SS_\beta(\cS)(d_{r-1})\big) = \SS_\mu(W), \qquad
\rH^{r-1}\big(\SS_\beta(\cS)(d_r)\big) = \SS_\lambda(W).
\end{align*}
In particular, the $\rE_1$ page contains only the map
\[
\SS_\lambda(V) \otimes \SS_\mu(W) \otimes R \leftarrow \SS_\mu(V) \otimes \SS_\lambda(W) \otimes R,
\]
with the left term located on the main diagonal. We see that $\rH^i(\cM) = 0$ for $i > 0$ and that the map above is a resolution of $\rH^0(\cM)$ by free $R$-modules.

\subsection{A stronger version of Theorem~\ref{thm:realizable}}

Use the notation of Theorem~\ref{thm:realizable}. Since $M$ is torsion, the ranks must agree,
\[
c_0 K_{\lambda^{(0)}}(k) = c_1 K_{\lambda^{(1)}}(k).
\]
A straightforward choice of $c_0, c_1$ is to take $c_0 = K_{\lambda^{(1)}}(k)$ and $c_1 = K_{\lambda^{(0)}}(k)$, and to look for a `small' resolution of the form 
\begin{equation}  \label{eqn:general-pure}
M \leftarrow \SS_{\lambda^{(0)}}(V) \otimes \SS_{\lambda^{(1)}}(W) \otimes R \leftarrow \SS_{\lambda^{(1)}}(V) \otimes \SS_{\lambda^{(0)}}(W) \otimes R\leftarrow 0,
\end{equation}
equivariant for both $\GL(V)$ and $\GL(W)$. Theorem~\ref{thm:linear-case} proves this conjecture when $|\lambda^{(1)}| = |\lambda^{(0)}|+1$ and one can also do the case when $\lambda^{(1)}$ is obtained by adding a connected border strip to $\lambda^{(0)}$ using Remark~\ref{rmk:generalizations}. In general, we do not know if this particular form of resolution exists, though we conjecture that it does:

\begin{conjecture}\label{conj:small-pure}
A small pure resolution \eqref{eqn:general-pure} exists, for any pair $\lambda^{(0)} \subsetneq \lambda^{(1)}$.
\end{conjecture}

We finish by establishing one more situation where Conjecture~\ref{conj:small-pure} is true:

\begin{proposition}\label{prop:special-small-resolution}
Suppose there exists $d$ so that $(\lambda^{(0)})_i \leq d \leq (\lambda^{(1)})_j$ for all $i,j$. $($Equivalently, assume $(\lambda^{(0)})_1 \leq (\lambda^{(1)})_k$.$)$ Then a small pure resolution \eqref{eqn:general-pure} exists.
\end{proposition}

\begin{proof}
We construct the map geometrically. After twisting down by $d$, we may suppose instead $\lambda^{(0)} \leq 0 \leq \lambda^{(1)}$. Write $\lambda^{(0)} = -\mu^R$ for some partition $\mu \geq 0$. On $X = \Hom(V,W)$, there is a canonical, bi-equivariant map of vector bundles $\cT \colon V \times X \to W \times X$, which is an isomorphism away from the determinant locus. When $\lambda \geq 0$, the Schur functor $\SS_\lambda(V)$ is functorial for linear transformations of $V$ (when $\lambda$ has negative parts, $\SS_\lambda$ is only functorial for isomorphisms), so there is an induced map
\[
\SS_{\lambda^{(1)}}(\cT) \colon \SS_{\lambda^{(1)}}(V) \times X \to \SS_{\lambda^{(1)}}(W) \times X,
\]
and, from the dual bundles, a second induced map
\[
\SS_\mu(\cT^*) \colon \SS_\mu(W^*) \times X \to \SS_\mu(V^*) \times X.
\]
Let $g = \SS_{\lambda_{(1)}}(\cT) \otimes \SS_\mu(\cT^*)$. Note that $g$ is generically an isomorphism of vector bundles, so the corresponding map of $R$-modules is injective:
\[
g \colon \SS_{\lambda^{(1)}}(V) \otimes \SS_\mu(W^*) \otimes R \to \SS_\mu(V^*) \otimes \SS_{\lambda^{(1)}}(W) \otimes R.
\]
Finally, we note that there is a canonical isomorphism of representations
$\SS_\mu(E^*) \cong \SS_{-\mu^R}(E)$
for any vector space $E$. Apply this to the free $R$-modules above to get the desired map.
\end{proof}

\bibliographystyle{alpha}
\bibliography{boijsoder-bib}{}

\begin{thebibliography}{BBEG12}

\bibitem[BBEG12]{BBEG2012}
Christine Berkesch, Jesse Burke, Daniel Erman, and Courtney Gibbons.
\newblock The cone of {B}etti diagrams over a hypersurface ring of low
  embedding dimension.
\newblock {\em J. Pure Appl. Algebra}, 216(10):2256--2268, 2012.
\newblock \arXiv{1109.5198v2}.

\bibitem[BS08]{BS2008}
Mats Boij and Jonas S{\"o}derberg.
\newblock Graded {B}etti numbers of {C}ohen-{M}acaulay modules and the
  multiplicity conjecture.
\newblock {\em J. Lond. Math. Soc. (2)}, 78(1):85--106, 2008.
\newblock \arXiv{math/0611081v2}.

\bibitem[BS12]{BS2012}
Mats Boij and Jonas S{\"o}derberg.
\newblock Betti numbers of graded modules and the multiplicity conjecture in
  the non-{C}ohen-{M}acaulay case.
\newblock {\em Algebra Number Theory}, 6(3):437--454, 2012.
\newblock \arXiv{0803.1645v1}.

\bibitem[Dev16]{sage}
The~Sage Developers.
\newblock {\em {S}ageMath, the {S}age {M}athematics {S}oftware {S}ystem}, 2016.
\newblock \url{http://www.sagemath.org}.

\bibitem[EE12]{EE2012}
David {Eisenbud} and Daniel {Erman}.
\newblock Categorified duality in {B}oij--{S}\"oderberg theory and invariants
  of free complexes.
\newblock {\em J. Eur. Math. Soc. (JEMS)}, to appear, 2012.
\newblock \arXiv{1205.0449v2}.

\bibitem[EFW11]{EFW2011}
David Eisenbud, Gunnar Fl{\o}ystad, and Jerzy Weyman.
\newblock The existence of equivariant pure free resolutions.
\newblock {\em Ann. Inst. Fourier (Grenoble)}, 61(3):905--926, 2011.
\newblock \arXiv{0709.1529v5}.

\bibitem[ES09]{ES2009}
David Eisenbud and Frank-Olaf Schreyer.
\newblock Betti numbers of graded modules and cohomology of vector bundles.
\newblock {\em J. Amer. Math. Soc.}, 22(3):859--888, 2009.
\newblock \arXiv{0712.1843v3}.

\bibitem[ES17]{questions-BS}
Daniel Erman and Steven~V Sam.
\newblock Questions about {B}oij--{S}\"oderberg theory.
\newblock In {\em Surveys on recent developments in algebraic geometry}, Proc.
  Sympos. Pure Math., pages 285--304. Amer. Math. Soc., 2017.
\newblock \arXiv{1606.01867v1}.

\bibitem[FL16]{FL16}
Nic Ford and Jake Levinson.
\newblock Foundations of {B}oij--{S}\"oderberg theory for {G}rassmannians.
\newblock 2016.
\newblock \arXiv{1609.03446v2}.

\bibitem[Fl{\o}12]{floystad-expository}
Gunnar Fl{\o}ystad.
\newblock Boij--{S}\"oderberg theory: Introduction and survey.
\newblock In {\em Progress in Commutative Algebra 1, Combinatorics and
  homology}, Proceedings in mathematics, pages 1--54. du Gruyter, 2012.
\newblock \arXiv{1106.0381v2}.

\bibitem[FMP16]{survey2}
Gunnar Fl{\o}ystad, Jason McCullough, and Irena Peeva.
\newblock Three themes of syzygies.
\newblock {\em Bull. Amer. Math. Soc.}, 53:415--435, 2016.

\bibitem[Ful96]{Fulton}
William Fulton.
\newblock {\em Young Tableaux}.
\newblock Cambridge University Press, 1996.
\newblock Cambridge Books Online.

\bibitem[GS]{M2}
Daniel~R. Grayson and Michael~E. Stillman.
\newblock Macaulay2, a software system for research in algebraic geometry.
\newblock Available at \url{http://www.math.uiuc.edu/Macaulay2/}.

\bibitem[GS16]{3pts}
Iulia Gheorghita and Steven~V Sam.
\newblock The cone of {B}etti tables over three non-collinear points in the
  plane.
\newblock {\em J. Commut. Algebra}, 8:537--548, 2016.
\newblock \arXiv{1501.00207v1}.

\bibitem[KS15]{kummini-sam}
Manoj Kummini and Steven~V Sam.
\newblock The cone of {B}etti tables over a rational normal curve.
\newblock In {\em Commutative Algebra and Noncommutative Algebraic Geometry},
  volume~68 of {\em Math. Sci. Res. Inst. Publ.}, pages 251--264. Cambridge
  Univ. Press, Cambridge, 2015.
\newblock \arXiv{1301.7005v2}.

\bibitem[LP09]{matchings}
L{\'a}szl{\'o} Lov{\'a}sz and Michael~D. Plummer.
\newblock {\em Matching theory}.
\newblock AMS Chelsea Publishing, Providence, RI, 2009.
\newblock Corrected reprint of the 1986 original.

\bibitem[SS12]{tca}
Steven~V Sam and Andrew Snowden.
\newblock Introduction to twisted commutative algebras.
\newblock 2012.
\newblock \arXiv{1209.5122v1}.

\bibitem[SW11]{sam-weyman}
Steven~V Sam and Jerzy Weyman.
\newblock Pieri resolutions for classical groups.
\newblock {\em J. Algebra}, 329:222--259, 2011.
\newblock \arXiv{0907.4505v5}.

\bibitem[Wey03]{weyman}
Jerzy Weyman.
\newblock {\em Cohomology of vector bundles and syzygies}.
\newblock Cambridge tracts in mathematics. Cambridge University Press,
  Cambridge, New York, 2003.

\end{thebibliography}

\end{document}